\documentclass[10pt]{article}
\usepackage{amsmath}
\usepackage{amsfonts}
\usepackage{amssymb}
\usepackage{amscd}
\usepackage{amsthm}
\usepackage{amsbsy}
\usepackage{graphicx}
\usepackage{bm}
\newtheorem{theorem}{Theorem}
\newtheorem{lemma}{Lemma}
\newtheorem{corollary}{Corollary}
\newtheorem{definition}{Definition}

\newtheorem*{RPMI}{The Riemannian Positive Mass Inequality}
\newtheorem*{RPI}{The Riemannian Penrose Inequality}
\newtheorem*{RZASI}{The Riemannian ZAS Inequality}
\newtheorem*{CC}{The Conformal Conjecture}

\newcommand{\real}{{\mathbb R}}

\newcommand{\ol}{\overline}

\date{October 23, 2010}

\title{
\bf On the Positive Mass, Penrose, and ZAS Inequalities in General Dimension}

\author{Hubert L. Bray\footnote{Mathematics Department, Duke University, Box 90320, Durham, NC 27708, USA, bray@math.duke.edu.  Supported in part by NSF grants DMS-0706794 and DMS-1007063.}}
\begin{document}

\maketitle


\thispagestyle{empty}

\begin{abstract}
\footnotesize{

\noindent After a detailed introduction including new
examples, we give an exposition focusing on the Riemannian cases of
the positive mass, Penrose, and ZAS inequalities of general
relativity, in general dimension.

\vspace*{2mm} \noindent{\bf 2000 Mathematics
Subject Classification:
 53C80, 53C24, 53C44.}

\vspace*{2mm} \noindent{\bf Keywords and Phrases: }General
relativity, scalar curvature, black holes, minimal surfaces, zero
area singularities, ZAS, positive mass inequality, Penrose inequality, ZAS inequality.}

\end{abstract}

\section{Dedication}

It is an honor and a pleasure to contribute a paper to this volume
celebrating Rick Schoen's 60th birthday (which doesn't seem possible
since Rick is so much better at every sport than the author). Rick's
contribution to mathematics continues to be tremendous and is
growing at an impressive rate, including his original research,
collaborations with others, and the many students whom he has so
expertly supervised.  As successful as he has been for so many
years, his continuing hard work and dedication reveals his genuine
passion and love for mathematics. He is a role model for us all.

The ideas and directions discussed in this paper were in large part
inspired by the positive mass inequality, proved in dimensions less
than eight in Rick's famous joint works with S.-T. Yau
\cite{PMT1,PMT2,SchoenYauCompact,SchoenVariational}. The positive mass inequality is a beautiful
geometric statement about scalar curvature which can equivalently be
understood as the fundamental result in general relativity that
positive energy densities in a spacetime imply that the total mass
of the spacetime is also positive.  In this paper we will discuss
the Riemannian cases of the positive mass inequality, the Penrose
inequality, and the ZAS (zero area singularity) inequality,
describing how they are closely related in section \ref{sec:trio}.
We apologize in advance for not making this a comprehensive survey of every interesting
result in these areas, although we do recommend \cite{Mars} as a very nice
survey of the Penrose inequality.
But first, in section \ref{sec:intro} we begin with a mixture of
well known facts and interesting new examples to motivate our later
discussion.

\section{Introduction}\label{sec:intro}

In this section we describe some of the ideas, motivations, and
definitions that are central to general relativity, in general
dimension.  We also discuss intuition and include some new examples
due to Lam \cite{LamThesis, LamPaper} of manifolds where the
Riemannian positive mass inequality and the Riemannian Penrose
inequality can be proved directly.

Broadly defined, general relativity encompasses any description of
the universe as a smooth manifold with a Lorentzian metric, called a
spacetime, of signature $(1,n)$, where usually $n = 3$. Then the
Einstein equation
\begin{equation}\label{eqn:EE}
   G = (n-1)\omega_{n-1} T
\end{equation}
can be taken to be the definition of the stress energy tensor $T$,
where $\omega_{n-1}$ is the measure of the unit $(n-1)$ sphere.
Thus, in dimension three, $G = 8\pi T$.  The Einstein curvature
tensor $G = \mbox{Ric} - \frac12 R \cdot g$ has zero divergence by
the second Bianchi identity. Furthermore, $T(\nu_1,\nu_2)$ is
defined to be the amount of energy traveling in the unit direction
$\nu_1$ as observed by someone going in the unit direction $\nu_2$.
We comment that the Einstein equation can be derived from an action
principle based on the Einstein-Hilbert action in the vacuum case,
but we will take the Einstein equation as our starting point. The
zero divergence property of $G$ (which is a consequence of the
action principle) may be interpreted as a local conservation
property for the stress energy tensor $T$.

\subsection{The Schwarzschild Spacetimes}\label{sec:ss}

The vacuum case of general relativity is $G = 0$ which, for $n \ge
2$, implies that $\mbox{Ric} = 0$. Unlike Newtonian physics which is
trivial in the vacuum case, general relativity is highly nontrivial
in the vacuum case when $n \ge 3$. That is, there exist many
solutions to $G = 0$ other than the Minkowski spacetime metric
\begin{equation}
   -dt \otimes dt + dx^1 \otimes dx^1 + ... + dx^n \otimes dx^n,
\end{equation}
which include solutions describing gravitational waves which
propagate at the speed of light.  However, if we restrict our
attention to spacetimes which are static and spatially spherically
symmetric, then it is an exercise to show that for $n \ge 3$ there
is precisely a one parameter family of vacuum solutions, called the
Schwarzschild spacetimes.  Explicitly, these spacetimes are
\begin{equation}\label{eqn:Schwarz}
   - \left( \frac{1-\frac{k}{r^{n-2}}}{1+\frac{k}{r^{n-2}}} \right)^2 dt \otimes dt +
   \left(1+\frac{k}{r^{n-2}}\right)^{\frac{4}{n-2}}
   \left( dx^1 \otimes dx^1 + ... + dx^n \otimes dx^n \right),
\end{equation}
for $r > |k|^{1/(n-2)}$, where $r = \sqrt{(x^1)^2 + ... + (x^n)^2}$.

Note that $k = 0$ corresponds to the flat Minkowski spacetime.
Furthermore, it is an easy exercise to show that all of the $k > 0$
spacetimes are scalings of each other, as are all of the $k < 0$
spacetimes.  Hence, up to scalings, equation \ref{eqn:Schwarz}
defines three special geometries, namely when $k = 1, 0, -1$.  As we
will see, each of these geometries inspires interesting physical
statements.

For $k > 0$, the spacetime in equation \ref{eqn:Schwarz} represents
the region outside a static black hole in vacuum.  Note that there
is a coordinate chart singularity at $r = k^{1/(n-2)}$ since the
first metric term goes to zero.  However, this value of $r$ is
purely a coordinate chart singularity and not a true metric
singularity.  Kruskal coordinates \cite{ONeill} can be used to
understand the entire Schwarzschild spacetime.  The above choice of
coordinate chart is limited in that it does not cover the interior
region of the Schwarzschild spacetime which contains the true metric
singularity, but it is ideal for our purposes as a way of
visualizing the Schwarzschild spacetime as a perturbation of the
Minkowski spacetime.

For $k < 0$, the spacetime in equation \ref{eqn:Schwarz} represents
the region outside a static zero area singularity, called a ZAS
\cite{BrayJauregui}. Again, there is a singularity at $r =
|k|^{1/(n-2)}$, but this time the singularity is a true metric
singularity as the curvatures of the spacetime diverge to infinity
as one approaches $r = |k|^{1/(n-2)}$.  Note that while
topologically the singularity is a coordinate $(n-1)$-sphere, the
area of this sphere is zero. Hence, to an observer in the spacetime,
the ZAS looks like a point singularity in this case.

Gravity, which is not a force in general relativity but instead is a
perceived phenomenon, results from postulating that test particles
in free fall follow geodesics.  A curve in a spacetime with
coordinates $\gamma(s) = (\gamma^0(s), \gamma^1(s), ...,
\gamma^n(s))$ is a geodesic if it satisfies the geodesic equation
$\nabla_{\dot\gamma} \dot\gamma = 0$, which in coordinates is
\begin{equation}
   0 = \ddot{\gamma}^\alpha(s) + \dot\gamma^i(s) \dot\gamma^j(s)
   \Gamma_{ij}^\alpha
\end{equation}
for all $\alpha$, where we denote time as the zeroth coordinate and
differentiation by $s$ with a dot. Hence, if we consider a curve
with $\gamma(0) = (0,...,0,r)$ and $\dot\gamma(0) = (1,0,...,0)$
representing a test particle a distance $r$ away from the origin in
the coordinate chart along the $n$th axis momentarily ``at rest''
going purely in the time direction, it follows from direct
calculation that $\ddot\gamma(0) = (0,...,0,-a)$, where $a =
\Gamma_{00}^n$ represents inward acceleration in the coordinate
chart, and
\begin{equation}
   \lim_{r \rightarrow \infty} a r^{n-1} = 2k(n-2)
\end{equation}
for the Schwarzschild spacetime. Hence, in the limit as $r$ goes to
infinity, the geodesic accelerates in the coordinate chart towards
the origin according to a $\frac{2k(n-2)}{r^{n-1}}$ inward
acceleration law.

Now suppose we generalize Newtonian gravity to $n$ dimensions by
defining the potential energy of a test particle divided by the mass
of the test particle a distance $r$ from a point mass of mass $m$ to
be $-\frac{Gm}{r^{n-2}}$ (where here $G$ is the universal
gravitational constant). This corresponds to an inward acceleration
law of $\frac{Gm(n-2)}{r^{n-1}}$.  In the same way that setting the
speed of light to one equates units of length and time, setting
$G=1$ equates units of length to the power $n-2$ to mass.

In order to recover Newtonian dynamics from the Schwarzschild
spacetime in the limit as $r$ goes to infinity, we set
\begin{equation}
   \frac{2k(n-2)}{r^{n-1}} = \frac{Gm(n-2)}{r^{n-1}}
\end{equation}
to derive $2k = Gm$.  Thus, setting $G=1$ implies that $k = m/2$.
Hence, when our one parameter family of Schwarzschild spacetimes is
expressed as
\begin{equation}\label{eqn:SchwarzMassM}
   - \left( \frac{1-\frac{m}{2r^{n-2}}}{1+\frac{m}{2r^{n-2}}} \right)^2 dt \otimes dt +
   \left(1+\frac{m}{2r^{n-2}}\right)^{\frac{4}{n-2}}
   \left( dx^1 \otimes dx^1 + ... + dx^n \otimes dx^n \right),
\end{equation}
Newtonian dynamics is recovered by the behavior of nonrelativistic
geodesics in the limit as $r$ goes to infinity, where $m$ is the
total mass of the spacetime.

At first glance this result seems paradoxical since the
Schwarzschild spacetimes are indeed vacuum solutions to the Einstein
equation.  How can they have nonzero total mass if they represent
vacuum solutions?  When $m > 0$, the answer is that these spacetimes
represent static black holes with event horizons at $r =
(m/2)^{{1}/{(n-2)}}$ (``coincidentally'' where the coordinate chart
singularity is) which contribute a mass $m$ to the system. The more
matter and gravitational waves that fall into a black hole, the most
massive it becomes.  Thus, while this is not precise, metaphorically
one could think of the mass of a black hole as resulting from the
mass and energy density that it has already consumed.  Similarly,
when $m<0$, it is reasonable to think of the ZAS singularity at the
spatial origin as a singularity with negative mass.  However,
whereas the singularity of the Schwarzschild spacetime with $m>0$ is hidden
inside the event horizon of a black hole, when $m<0$ the singularity of the ZAS
is exposed for all observers to see.  The time evolution of a ZAS
has yet to be defined and it is not clear that it represents anything
physical. However, geometrically ZAS are quite interesting and are very natural
phenomena to study.

\subsection{The Schwarzschild Metrics}

We define the Schwarzschild metric to be the Riemannian manifold
isometric to the $t=0$ hypersurface of the Schwarzschild spacetime.
Explicitly, the Schwarzschild metric of mass $m$ in dimension $n \ge
3$ is
\begin{equation}\label{eqn:conformalform}
   \left( \real^n \backslash \bar{B},
   \left(1+\frac{m}{2r^{n-2}}\right)^{\frac{4}{n-2}}
   \left( dx^1 \otimes dx^1 + ... + dx^n \otimes dx^n \right) \right),
\end{equation}
where $\bar{B}$ is the closed ball of radius
$\max(0,-m/2)^{1/(n-2)}$. Technically we have extended this
Schwarzschild metric when $m>0$ since we are now allowing $r$ to go
all the way down to zero.  This is justified by the fact that for
$m>0$, the above Riemannian metric is isometric to the $t=0$ slice
of the Schwarzschild spacetime in Kruskal coordinates (which covers
both exterior and both interior regions of the Schwarzschild
spacetime) \cite{ONeill}.

Again, all of the $m>0$ Schwarzschild metrics are scalings of each
other, as are all of the $m<0$ Schwarzschild metrics.  Hence, up to
scalings, the Schwarzschild metrics define three important
geometries to study, namely when $m/2 = 1,0,-1$.  In particular,
these three geometries are the cases of equality of the Riemannian
Penrose inequality, the Riemannian positive mass inequality, and the
Riemannian ZAS inequality (this last one modulo an open geometric
conjecture), respectively, which we will discuss in detail later.

The conformal description of the Schwarzschild metrics given in
equation \ref{eqn:conformalform} are a great way to gain some
intuition about these metrics and, more generally, the notion of
total mass.  Recall that the parameter $m$ was chosen in our
discussion of the Schwarzschild spacetime metrics to be physically
compatible with our usual notion of total mass from Newtonian
physics by looking at the accelerations of geodesics with respect to
asymptotically flat coordinate charts.  Even though the time
dimension has been removed in the Schwarzschild metrics, the total
mass $m$ is still present as the coefficient of the
$\frac{1}{r^{n-2}}$ term in the expression of the metric.  Hence,
the total mass is directly related to the rate at which the
Schwarzschild metric becomes flat at infinity.

Another important geometric characterization of the Schwarzschild
metrics is that they are all of the spherically symmetric metrics
which have zero scalar curvature.  Hence, these metrics are important
for understanding scalar curvature.  As we will see later, scalar
curvature is proportional to local energy density for zero second
fundamental form slices of spacetimes. Hence, the fact that the
Schwarzschild spacetime is vacuum implies that the Schwarzschild
metric must have zero scalar curvature.  This is also easily
verified by direct calculation.

The Schwarzschild metrics of dimension $n$ have natural embeddings
into flat spaces of dimension $n+1$, which is useful for another
view of these geometries. When $m>0$, the $n$ dimensional
Schwarzschild metric may be conveniently viewed as a spherically
symmetric hypersurface of $n+1$ dimension Euclidean space (with the
usual positive definite Euclidean metric) \cite{BrayThesis}, as
shown in figure \ref{PositiveSchwarzschild} when $n=3$.  Based on
this, Greg Galloway noted that when $m<0$, one may do a Wick
rotation where $w$ is replaced by $iw$ to conclude that for $m<0$,
the $n$ dimensional Schwarzschild metric is isometric to a
spherically symmetric hypersurface of $n+1$ dimension Lorentzian
space, as shown in figure \ref{NegativeSchwarzschild} when $n=3$.

These two figures are qualitatively correct in higher dimensions as
well, although $R$ as a function of $w$ is not quite as explicit.
Instead, we comment that for $m>0$, $\left(\frac{dR}{dw}\right)^2 =
\frac{R^{n-2}}{2m} - 1$ gives the correct embedding in Euclidean
space, and for $m<0$, $\left(\frac{dR}{dw}\right)^2 =
\frac{R^{n-2}}{2|m|} + 1$ gives the correct embedding in Lorentzian
space, for $n \ge 3$.

\begin{figure}
   \begin{center}
   \includegraphics[width=120mm]{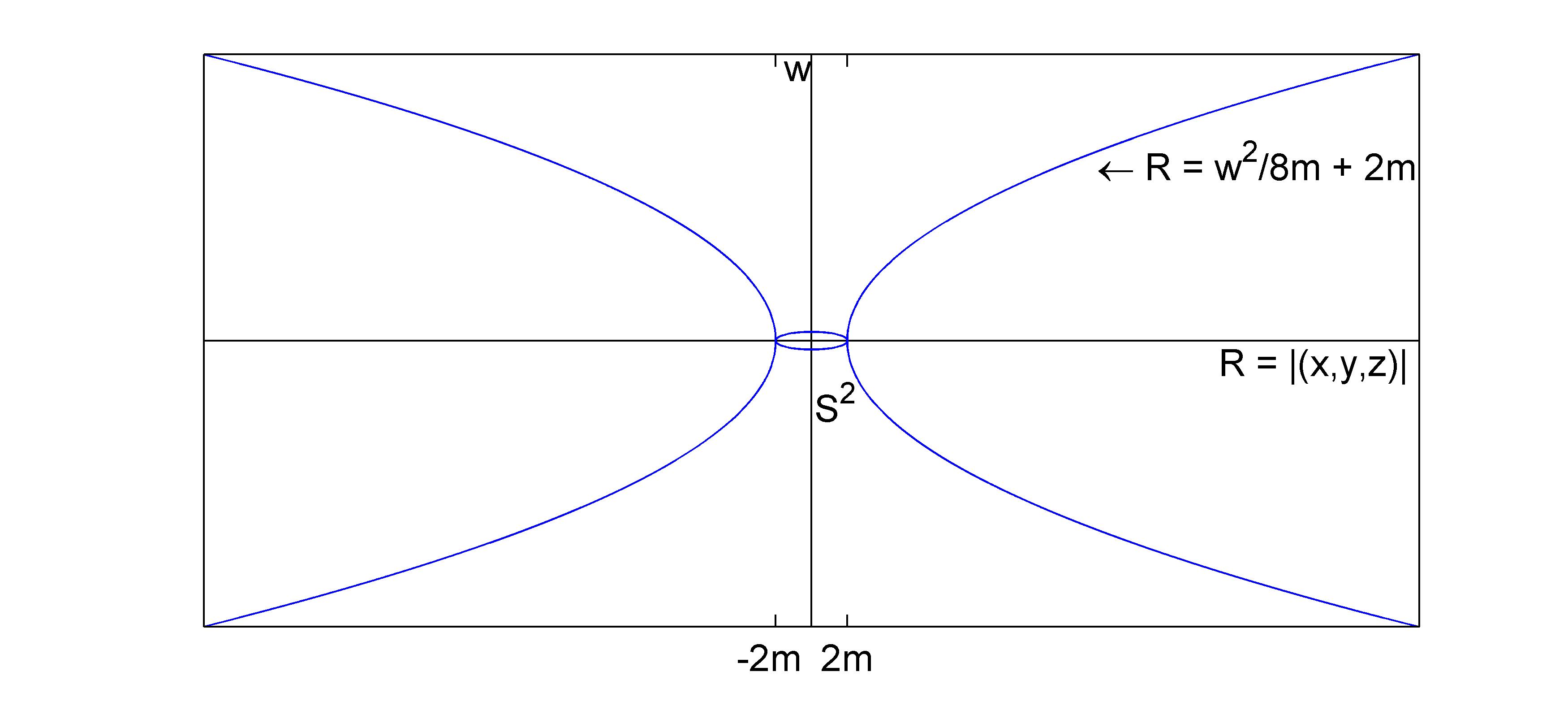}
   \end{center}
   \caption{The three dimensional Schwarzschild metric of mass $m>0$
   (in blue)
   viewed as a spherically symmetric submanifold of four dimensional
   Euclidean space \{(x,y,z,w)\} satisfying $R = w^2/8m + 2m$, where
   $R = \sqrt{x^2 + y^2 + z^2}$.\label{PositiveSchwarzschild}}
\end{figure}

\begin{figure}
   \begin{center}
   \includegraphics[width=120mm]{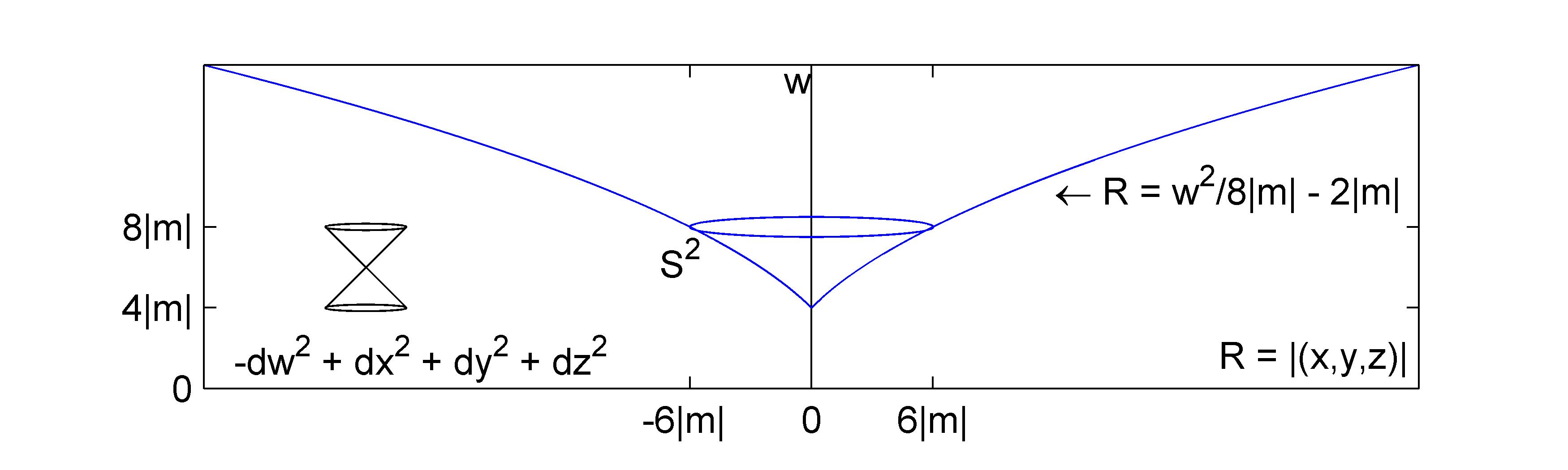}
   \end{center}
   \caption{The three dimensional Schwarzschild metric of mass $m<0$
   (in blue)
   viewed as a spherically symmetric submanifold of four dimensional
   Minkowski space \{(x,y,z,w)\} (with Lorentzian metric $-dw^2 + dx^2 + dy^2 + dz^2$)
   satisfying $R = w^2/8|m| - 2|m|$, where
   $R = \sqrt{x^2 + y^2 + z^2}$.\label{NegativeSchwarzschild}}
\end{figure}

When $m>0$, figure \ref{PositiveSchwarzschild} shows that the
Schwarzschild metric also has a $Z_2$ symmetry by reflecting through
the $w=0$ plane.  The set of fixed points of this symmetry is a
minimal $(n-1)$ sphere of $(n-1)$ volume $A = \omega_{n-1}
(2m)^{\frac{n-1}{n-2}}$, where $\omega_{n-1}$ is the measure of the
unit $(n-1)$ sphere.  Going back to the conformal description of the
Schwarzschild metrics in equation \ref{eqn:conformalform}, this
minimal sphere is realized at $r_0 = (m/2)^{1/(n-2)}$.  Furthermore,
in the conformal picture, the $Z_2$ isometry is reflection through
this sphere in the conformal coordinate chart, which explicitly is
$r \rightarrow r_0^2 / r$. Hence, zero goes to infinity and vice
versa, highlighting the nonobvious fact that there is another
asymptotically flat end in the neighborhood of $r=0$ in the
conformal coordinate chart of equation \ref{eqn:conformalform}.

When $m<0$, figure \ref{NegativeSchwarzschild} gives a nice picture
of the singularity.  We comment that the nature of this singularity
is complicated by the fact that the graph is becoming null at the
singularity as well.  We refer the reader to \cite{BrayJauregui} for
more discussion.

\subsection{The Total Mass of Asymptotically Flat Manifolds}

The Schwarzschild metrics give us examples of asymptotically flat
manifolds where a notion of total mass is well defined.  A logical
next step is to find the most general notion of an asymptotically
flat manifold for which a physically relevant definition of total
mass can be defined.  We refer the reader to papers by Bartnik who
studied this problem in detail \cite{Bartnik}. The most commonly
used definitions for an asymptotically flat manifold $M^n$, $n \ge
3$, and its total mass $m$, are given below.

\begin{definition}\cite{Bartnik}\label{def:af}
    A complete Riemannian manifold $(M^n,g)$ of dimension $n$ is said to be {\bf asymptotically flat}
    if there is a compact subset $K\subset M$, such that $M\backslash K$ is diffeomorphic to
    $\mathbb R^n\backslash\{|x|\leq 1\}$, and a diffeomorphism
    $\Phi:M\backslash K\to \mathbb R^n\backslash\{|x|\leq 1\}$ such that, in the coordinate chart
    defined by $\Phi$, $g=g_{ij}(x)dx^idx^j$, where
\begin{align*}
    g_{ij}(x) &= \delta_{ij}+O(|x|^{-p}) \\
    |x||g_{ij,k}(x)|+|x|^2|g_{ij,kl}(x)| &= O(|x|^{-p}) \\
    |R(g)(x)| &= O(|x|^{-q})
\end{align*}
for some $q>n$ and $p>(n-2)/2$.
\end{definition}

\begin{definition}\cite{ADM,Bartnik}\label{def:adm}
    The total mass $m$ of a complete, asymptotically flat manifold $(M,g)$ is defined to be
\begin{equation}\label{eqn:massdef}
    m = \lim_{r\to\infty} \frac{1}{2(n-1)\omega_{n-1}} \int_{S_r} \sum_{i,j} (g_{ij,i}-g_{ii,j}) \nu_j dS_r,
\end{equation}
where $\omega_{n-1}$ is the volume of the $(n-1)$ unit sphere, $S_r$
is the coordinate sphere of radius $r$, $\nu$ is the outward unit
normal to $S_r$ and $dS_r$ is the area element of $S_r$ in the
coordinate chart.
\end{definition}

The above definition for an asymptotically flat manifold is easily
generalized to allow for more than one asymptotically flat end,
where each connected component of $M \backslash K$, referred to as
an end, is required to satisfy the stated asyptotically flat
condition. For example, each Schwarzschild metric with $m>0$ has two
asymptotically flat ends, as can be seen in figure
\ref{PositiveSchwarzschild}. The definition of total mass then
applies to each end separately, with each end having its own total
mass.

The above definition for the total mass, also known as the ADM mass,
was defined in \cite{ADM}.  (Definitions of total mass sometimes
differ from the above by a factor, corresponding to different
conventions than the one defined in section \ref{sec:ss}.) Later,
Bartnik showed that the above definition is well-defined by showning
that it is independent of the choice of asymptotically flat
coordinates \cite{Bartnik}. Finally, the above definition of total
mass is compatible with our physically derived Schwarzschild metrics
with metric components $g_{ij} = \left(1 + m/2r^{n-2}
\right)^{4/(n-2)} \delta_{ij}$, which is easily verified by direct
calculation.

\subsection{Local Energy Density and Scalar Curvature}

Suppose that $(M^n,g)$ is a Riemannian n-manifold isometrically
embedded in an (n+1) dimensional Lorentzian spacetime. Suppose that
$M^n$ has zero second fundamental form in the spacetime.  This is a
simplifying assumption which allows us to think of $(M^n,g)$ as a
``$t=0$'' slice of the spacetime.  Then by the Einstein equation
(equation \ref{eqn:EE}), the energy density $\mu$ as seen by an
observer moving orthogonally to the hypersurface $M^n$ in the unit
time direction $\nu$ is
\begin{equation}\label{eqn:energydensity}
   \mu = T(\nu,\nu) = \frac{1}{(n-1)\omega_{n-1}} G(\nu,\nu) = \frac{R}{2(n-1)\omega_{n-1}},
\end{equation}
where the last equality follows from the Gauss equation \cite{ONeill} and $R$ is
the scalar curvature of $(M^n,g)$.

Equation \ref{eqn:energydensity} reveals an important connection
between energy density and scalar curvature, namely that they are
the same, up to a factor, in the context of zero second fundamental
form space-like hypersurfaces of spacetimes.  This equation is the
reason that the Riemannian positive mass inequality, the Riemannian
Penrose inequality, and the Riemannian ZAS inequality, discussed in
the next section, are not just important physical statements
concerning energy density and total mass, but are also fundamental
geometric statements about scalar curvature.  More specifically, the
physical assumption of nonnegative energy density everywhere implies
that $(M^n,g)$ must have nonnegative scalar curvature.  Hence, the
geometric study of asymptotically flat manifolds with nonnegative
scalar curvature has physical implications.

\subsection{Example:  Conformally Flat Manifolds}

Consider asymptotically flat manifolds of the form $(\real^n,g)$, $n
\ge 3$, where
\begin{equation}\label{eqn:confmetric}
   g = u(x)^{\frac{4}{n-2}} \delta,
\end{equation}
where $\delta$ is the flat metric on $\real^n$ and $u(x) > 0$ is
$C^2$ and goes to one at infinity in a manner satisfying the
asymptotically flat condition.  Then direct computation shows that
\begin{equation}\label{eqn:confscalarcurv}
   R = -\frac{4(n-1)}{n-2} u(x)^{-\frac{n+2}{n-2}} \Delta u,
\end{equation}
where $\Delta$ is the Laplacian operator of $(\real^n,\delta)$. (The
exponent $\frac{4}{n-2}$ in equation \ref{eqn:confmetric} is chosen
to make equation \ref{eqn:confscalarcurv} particularly nice.) Hence,
nonnegative energy density implies that our manifold must have
nonnegative scalar curvature everywhere, which in turn implies that
$u(x)$ must be superharmonic.

Also, $R=0$ in an open region $\Omega$ implies that $u(x)$ is
harmonic in $\Omega$.  Furthermore, the spherically symmetric
harmonic functions in $\real^n$ are linear combinations of $1$ and
$\frac{1}{r^{n-2}}$.  Hence, if we require the asymptotically flat
manifold to be spherically symmetric with zero scalar curvature
outside a finite radius, then it follows that
\begin{equation}\label{eqn:asySchwarz}
   u(x) = 1 + \frac{m}{2r^{n-2}}
\end{equation}
there, for some $m$.  Hence, the manifold will be isometric to a
Schwarzschild metric outside the finite radius.  This is a
convenient assumption to make since then the total mass is just $m$.
Then
\begin{eqnarray}
   m &=& \lim_{r \rightarrow \infty} \frac{2}{(n-2)\omega_{n-1}} \int_{S_r}
       - \langle \nabla u, \nu \rangle dS_r \\
     &=& \frac{2}{(n-2)\omega_{n-1}} \int_{M^n} -\Delta u \;dV_\delta \\
     &=& \frac{2}{(n-2)\omega_{n-1}} \int_{M^n} \frac{n-2}{4(n-1)} R u(x)^{\frac{n+2}{n-2}} \;dV_\delta \\
     &=&  \int_{M^n} \frac{R}{2(n-1)\omega_{n-1}} u(x)^{\frac{2-n}{n-2}}
     \;dV_g \label{eqn:intscalarcurv}
\end{eqnarray}
where $\nu$ is the outward unit normal to $S_r$, $dS_r$ is the area
form of $S_r$, $dV_\delta$ is the volume form of $(M^n,\delta)$, and
$dV_g$ is the volume form of $(M^n,g)$. The first equality follows
from the special form of $u(x)$ given by equation
\ref{eqn:asySchwarz}, although not surprisingly it is true for all
conformally flat asymptotically flat manifolds.  The second equality
is the divergence theorem.  The third step follows from equation
\ref{eqn:confscalarcurv} and the fourth step follows from $dV_g =
u(x)^\frac{2n}{n-2} dV_\delta$.

Now suppose that we let $u(x) = 1 + \epsilon v(x)$, where $v(x)$ is
$C^2$, superharmonic, and equal to $\frac{1}{2r^{n-2}}$ outside a
finite radius.  Then in the limit as $\epsilon$ goes to zero, $u(x)
\approx 1$ and
\begin{equation}\label{eqn:approx}
   m \approx  \int_{M^n} \frac{R}{2(n-1)\omega_{n-1}} \;dV_g,
\end{equation}
by which we mean that the ratio of the left hand side to the right
hand side of this equation equals one in the limit as $\epsilon$
goes to zero.

Hence, equation \ref{eqn:approx} implies that the total mass is
approximately the integral of energy density (as defined in equation
\ref{eqn:energydensity}) in this example.  This is what one would
expect from Newtonian physics. Equivalently, if we take this model
to be the physically relevant perturbation of Minkowski space and
its flat $\real^3$ zero second fundamental form hypersurface
corresponding to the Newtonian limit (a discussion which we will
skip), this argument defines the correct proportionality constant in
the Einstein equation (equation \ref{eqn:EE}).

However, especially as we get further from the $\epsilon = 0$ limit,
it is important to note that total mass is not the integral of
energy density. Since $u(x)$ is superharmonic, the maximum principle
implies that $u(x) \ge 1$. Hence, in this particular example,
\begin{equation}
   m \le  \int_{M^n} \frac{R}{2(n-1)\omega_{n-1}} \;dV_g,
\end{equation}
which roughly corresponds to the fact that Newtonian potential
energy, which for two positive point masses is $- m_1 m_2 /
r^{n-2}$, is negative. However, there are also examples of
asymptotically flat manifolds where
\begin{equation}
   m \ge  \int_{M^n} \frac{R}{2(n-1)\omega_{n-1}} \;dV_g.
\end{equation}
If fact, given any asymptotically flat manifold with nonnegative
scalar curvature, one can multiply it by a conformal factor to achieve
zero scalar curvature everywhere.  By the Riemannian positive mass
inequality (which currently requires $n \le 7$ or $M^n$ to be spin)
discussed in the next section, the total mass of the resulting zero
scalar curvature manifold will be positive, unless the resulting
manifold is the flat metric on $\real^n$. Hence, the connection
between local energy density and total mass is highly nontrivial.

The Riemannian positive mass inequality simply states that
nonnegative energy densities should imply that the total mass is
also nonnegative, with zero total mass only for the flat metric on
$\real^3$.  Physically, if this were not the case, then the local
condition of nonnegative energy density would not be enough to imply
that the total mass of an isolated body was also nonnegative. Hence,
the positive mass inequality shows that general relativity with a
nonnegative energy density condition is consistent with the fact
that objects of negative total mass are not observed in the
universe.

Note that equation \ref{eqn:intscalarcurv} proves the Riemannian
positive mass inequality for asymptotically flat manifolds conformal
to the flat metric on $\real^n$.  Clearly $R \ge 0$ implies $m \ge
0$, with $m = 0$ only when $R = 0$ everywhere, which by equation
\ref{eqn:confscalarcurv} implies that $u=1$ everywhere.

\subsection{Harmonically Flat Manifolds}

While definitions \ref{def:af} and \ref{def:adm} are important for
extending the notions of asymptotically flat and total mass to as
many manifolds as possible, for many purposes one can take a much
simpler view of what total mass is.  In \cite{PMT2}, Schoen and Yau
show that for any asymptotically flat manifold with $R \ge 0$, and
for any $\epsilon > 0$, one can perturb the manifold while
maintaining $R \ge 0$ such that the metric changes by less than
$\epsilon$ pointwise and the total mass changes by less than
$\epsilon$ as well, where the new perturbed manifold is
``harmonically flat.''

\begin{definition}\cite{PMT2,BrayPenrose}\label{def:hf}
    A complete Riemannian manifold $(M^n,g)$ of dimension $n$ is said to be {\bf harmonically flat}
    if there is a compact subset $K\subset M$, such that $M\backslash K$ is
    isometric to
\begin{equation*}
    (\real^n \backslash B_r(0), u(x)^{\frac{4}{n-2}}
    \delta_{ij})
\end{equation*}
     for some $r>0$ and some $u(x)$ which goes to one
    at infinity and is harmonic in $(\real^n \backslash B_r(0),
    \delta_{ij})$.
\end{definition}

The above definition is for manifolds with one asymptotically flat
end.  For manifolds with more than one end, each end is required to
have the above asymptotics.  Since harmonic functions in $\real^n$
are smooth, this imposes very nice asymptotics on harmonically flat
manifolds.  Furthermore, harmonic functions may be expanded in terms
of spherical harmonics, for which the first two terms are
\begin{equation}
   u(x) = 1 + \frac{m}{2r^{n-2}} +
   \mathcal{O}\left(\frac{1}{r^{n-1}}\right).
\end{equation}
Hence, the total mass is revealed as simply the coefficient of the
first nontrivial term in the spherical harmonic expansion.

This idea of perturbing the manifold to make the asymptotics as nice
as possible can be pushed even further, as noted in
\cite{BrayThesis}.  For any asymptotically flat manifold with $R \ge
0$, and for any $\epsilon > 0$, one can perturb the manifold while
maintaining $R \ge 0$ such that the metric changes by less than
$\epsilon$ pointwise and the total mass changes by less than
$\epsilon$ as well, where the new perturbed manifold is
``Schwarzschild at infinity.''

\begin{definition}\cite{BrayThesis}\label{def:schwarzatinfinity}
    A complete Riemannian manifold $(M^n,g)$ of dimension $n$ is said to be {\bf Schwarzschild at infinity}
    if there is a compact subset $K\subset M$, such that $M\backslash K$ is
    isometric to an end of a Schwarzschild metric
\begin{equation*}
    (\real^n \backslash B_r(0), \left( 1 + \frac{m}{2r^{n-2}}\right)^{\frac{4}{n-2}}
    \delta_{ij})
\end{equation*}
     for some $m$ and some $r>0$.
\end{definition}

As usual, manifolds with more than one end are required to have each
end isometric to an end of a Schwarzschild metric.  Hence, the above
definition shows that we can require the best asymptotics possible.
Furthermore, the above perturbation theorems imply that if you can
prove theorems like the Riemannian positive mass inequality, the
Riemannian Penrose inequality, or the Riemannian ZAS inequality for
manifolds which are Schwarzschild at infinity, then you can prove
them for asymptotically flat manifolds as well.  Thus, as far as the
picture one should have in one's head when considering
asymptotically flat manifolds, for many purposes one may as well
assume that they are Schwarzschild at infinity.

\subsection{New Example:  Graphs Over $\real^n$}

In this section we present a final set of examples, recently
discovered by George Lam \cite{LamThesis,LamPaper}.  In these
examples, the asymptotically flat manifold is assumed to be the
graph of a real-valued function over $\real^n$.  The Riemannian
positive mass inequality and the Riemannian Penrose inequality can
then be proven in these cases.  We refer the reader to
\cite{LamThesis,LamPaper} for more discussion and related results, including
analogous cases where graphs over $\real^n$ in $n+1$ dimensional
Minkowski space can be proven to satisfy the Riemannian ZAS
inequality in an manner similar to what we describe here.

\begin{figure}
   \begin{center}
   \includegraphics[width=120mm]{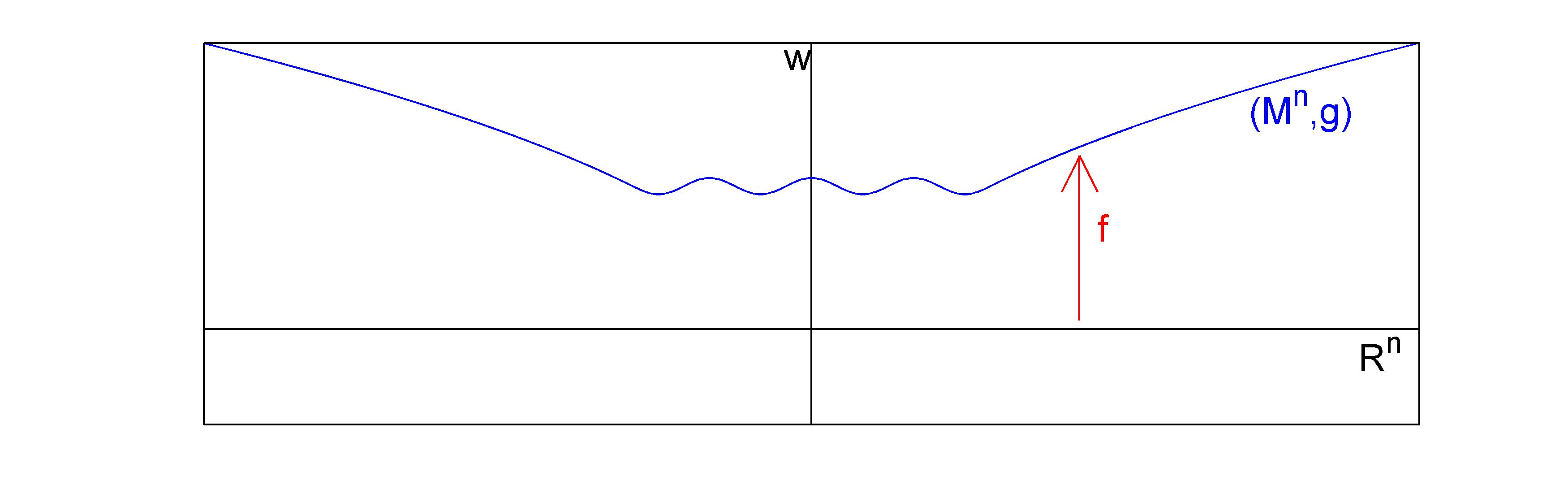}
   \end{center}
   \caption{The graph of a real-valued function $f$ defined on $\real^n$ has a natural metric, namely the
   one induced from the flat background metric on $\real^{n+1}$.  Lam's work provides an explicit formula for the total mass
   of these Riemannian manifolds as an integral involving scalar curvature and $df$.  \label{GraphOverRn}}
\end{figure}

Given a smooth function $f:\mathbb R^n\to\mathbb R$, the graph of
$f$ with metric induced from the flat metric on $\real^{n+1}$ is a
complete Riemannian manifold isometric to $(\mathbb R^n,
\delta_{ij}+f_if_j)$, as sketched in figure \ref{GraphOverRn}. The
following definition for an asymptotically flat function $f$ implies
that the resulting graph manifold is asymptotically flat.

\begin{definition}\cite{LamThesis, LamPaper}
    Let $f:\mathbb R^n\to\mathbb R$ be a smooth function and let $f_i$ denote the $i$th partial derivative of $f$. We say that $f$ is {\bf asymptotically flat} if
\begin{align*}
    f_i(x) &= O(|x|^{-p/2}) \\
    |x||f_{ij}(x)|+|x|^2|f_{ijk}(x)| &= O(|x|^{-p/2})
\end{align*}
at infinity for some $p>(n-2)/2$.
\end{definition}

Lam is then able to derive an explicit expression for the total mass
in terms of the scalar curvature of the graph, similar to equation
\ref{eqn:intscalarcurv} derived in the conformally flat manifold
example.

\begin{theorem}[Riemannian Positive mass inequality for Graphs over $\mathbb
R^n$]\cite{LamThesis,LamPaper}\label{thm:PMTGraphCase} Let $(M^n,g)$
be the graph of a smooth asymptotically flat function $f:\mathbb
R^n\to\mathbb R$ with the induced metric from $\mathbb R^{n+1}$. Let
$R$ be the scalar curvature and $m$ the total mass of $(M^n,g)$. Let
$dV_g$ denote the volume form on $(M^n,g)$. Then
\begin{equation}
    m=\int_{M^n} \frac{R}{2(n-1)\omega_{n-1}} \cdot \frac{1}{\sqrt{1+|\nabla f|^2}}\; dV_g.
\end{equation}
In particular, $R\geq 0$ implies $m\geq 0$.
\end{theorem}

As in the conformally flat manifold example, the total mass is
approximately the integral of scalar curvature divided by
$2(n-1)\omega_{n-1}$ for perturbations of the flat metric
corresponding to $df = 0$.  More generally, the total mass is less
than this integral, roughly corresponding to the fact that potential
energy is negative in Newtonian mechanics, as we also saw in the
conformally flat example.

However, unlike the conformal examples, these graph examples can be
extended to manifolds with boundaries to prove the Riemannian
Penrose inequality in many graph cases as well, sketched in figure
\ref{PenroseGraphCase}. To achieve this, the first step is the
following theorem, which we note still gives an explicit expression
for the total mass in terms of the scalar curvature and a boundary
term.

\begin{figure}
   \begin{center}
   \includegraphics[width=120mm]{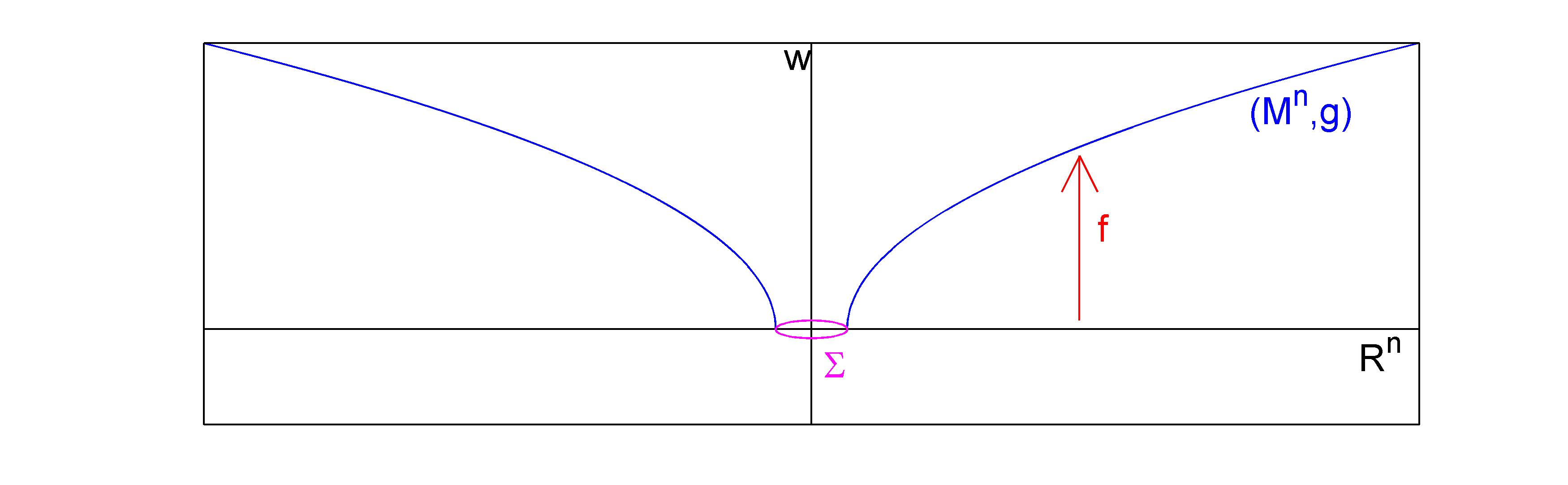}
   \end{center}
   \caption{Same example as in the previous figure, but now the graph is allowed to have a vertical boundary.
   Lam's previous result generalizes by gaining a boundary term, leading to a proof of the Riemannian Penrose inequality
   for these graph cases.  \label{PenroseGraphCase}}
\end{figure}

\begin{theorem}\cite{LamThesis,LamPaper}
Let $\Omega$ be a bounded and open (but not necessarily connected)
set with smooth boundary $\Sigma$ in $\mathbb R^n$. Let $f: \mathbb
R^n \backslash \Omega \to \mathbb R$ be a smooth asymptotically flat
function such that $\Sigma \subset f^{-1}(0)$ and $|\nabla
f(x)|\to\infty$ as $x\to\Sigma$. Let $(M^n,g)$ be the graph of $f$
with the induced metric from $(\mathbb R^n \backslash \Omega) \times
\mathbb R$. Let $H$ be the mean curvature of $\Sigma$ in $(\mathbb
R^n, \delta)$ and $d\Sigma$ be the area form on $\Sigma$. Then the
total mass of $(M^n,g)$ is
\begin{equation}
    m =  \int_{\Sigma} \frac{H d\Sigma}{2(n-1)\omega_{n-1}} +
     \int_{M^n} \frac{R}{2(n-1)\omega_{n-1}} \cdot \frac{1}{\sqrt{1+|\nabla f|^2}} \; dV_g.
\end{equation}
\end{theorem}
The next step is to use a special case of the Aleksandrov-Fenchel
inequality.
\begin{lemma}\cite{Schneider}
If $\Sigma$ is a convex surface in $\mathbb R^n$ with mean curvature
$H$ and area $|\Sigma|$, then
\begin{equation}\label{AFI}
    \frac{1}{2(n-1)\omega_{n-1}} \int_{\Sigma} H d\Sigma \geq \frac 12 \left( \frac{|\Sigma|}{\omega_{n-1}} \right) ^{\frac{n-2}{n-1}}.
\end{equation}
\end{lemma}
Putting the above two results together, Lam proves the Riemannian
Penrose inequality in this case.
\begin{corollary}[Riemannian Penrose Inequality for Graphs on $\mathbb R^n$ with Convex
Boundaries] \cite{LamThesis,LamPaper} With the same hypotheses as in
Theorem 2, and the additional assumption that each connected
component $\Omega_i$ of $\Omega$ is convex, then
\begin{equation}
    m \geq \sum_{i=1}^k \frac 12 \left( \frac{|\Sigma_i|}{\omega_{n-1}} \right) ^{\frac{n-2}{n-1}} +
     \int_{M^n} \frac{R}{2(n-1)\omega_{n-1}} \cdot \frac{1}{\sqrt{1+|\nabla f|^2}}\; dV_g,
\end{equation}
where $\Sigma_i = \partial \Omega_i$.  In particular, $R \ge 0$
implies that
\begin{equation}\label{eqn:RPIGraphCase}
    m \geq \sum_{i=1}^k \frac 12 \left( \frac{|\Sigma_i|}{\omega_{n-1}} \right)
    ^{\frac{n-2}{n-1}}.
\end{equation}
\end{corollary}

Equation \ref{eqn:RPIGraphCase} is actually \textit{stronger} than
the Riemannian Penrose inequality in general dimension, which is
\begin{equation}
    m \geq  \frac 12 \left( \frac{\sum_{i=1}^k |\Sigma_i|}{\omega_{n-1}} \right)
    ^{\frac{n-2}{n-1}}.
\end{equation}
The inequality in equation \ref{eqn:RPIGraphCase} is not true in
general and has explicit counterexamples.  It is quite interesting
that the ``graph over $\real^n$'' case is sufficiently restrictive
to give this stronger result.

In fact, as far as the author is aware, the above theorem is all of
the cases where the Riemannian Penrose inequality is known to be
true when $n \ge 8$ since the graph cases include the spherically
symmetric examples, which were the only cases known to the author
prior to the above result.  The general case of the Riemannian
Penrose inequality is known in dimension $n \le 7$
\cite{BrayPenrose, BrayLee}, as discussed in the next section.

The key identity which Lam derived to prove the above results is a
clever expression for the scalar curvature of graphs as the
divergence of a vector field in the $\real^n$ base space.  Direct
computation shows that
\begin{equation}\label{eqn:directcomp}
    R= \frac{1}{1+|\nabla f|^2} \left[ f_{ii}f_{jj} - f_{ij}f_{ij} - \frac{2f_jf_k}{1+|\nabla f|^2} (f_{ii}f_{jk} - f_{ij}f_{ik})
    \right],
\end{equation}
where subscripts denote partial differentiation in $\real^n$ and
$\nabla$ is the gradient in $\real^n$.  Denoting divergence in
$\real^n$ as $\nabla \cdot$, Lam then proves the following result.
\begin{lemma} \label{LamLemma}\cite{LamThesis, LamPaper}
The scalar curvature of the graph $(\mathbb R^n,\delta_{ij}+f_if_j)$
is
\begin{equation}
  R = \nabla \cdot \left( \frac{1}{1+|\nabla f|^2} (f_{ii}f_j - f_{ij}f_i)\partial_j \right).
\end{equation}
\end{lemma}
\begin{proof}
\begin{eqnarray*}
 & & \nabla \cdot \left( \frac{1}{1+|\nabla f|^2} ( f_{ii}f_j - f_{ij}f_i ) \partial_j \right) \\
 &=& \frac{1}{1+|\nabla f|^2} ( f_{iij}f_j + f_{ii}f_{jj} - f_{ijj}f_i - f_{ij}f_{ij} ) - \frac{2f_{jk}f_k}{(1+|\nabla f|^2)^2} (f_{ii}f_j - f_{ij}f_i) \\
 &=& \frac{1}{1+|\nabla f|^2} \left[ f_{ii}f_{jj} - f_{ij}f_{ij} - \frac{2f_jf_k}{1+|\nabla f|^2} (f_{ii}f_{jk} - f_{ij}f_{ik}) \right]\\
 &=& R
\end{eqnarray*}
by \eqref{eqn:directcomp}.
\end{proof}

The key point is that the scalar curvature may be expressed as a
divergence of a vector field, which is far from obvious from
equation \ref{eqn:directcomp}.  This allows a divergence theorem
argument to prove the desired results.  We refer the reader to
\cite{LamThesis,LamPaper} for the case where there is an interior boundary
which leads to the Riemannian Penrose inequality.  That case is a
natural generalization of the case where there is no interior
boundary.  The no interior boundary case is used by Lam to prove the
Riemannian positive mass inequality for manifolds which are graphs
over $\real^n$, shown below.
\begin{proof}[Proof of Theorem \ref{thm:PMTGraphCase}]
By definition, the total mass of $(M^n,g)=(\mathbb
R^n,\delta_{ij}+f_if_j)$ is
\begin{align*}
    m &= \lim_{r\to\infty} \frac{1}{2(n-1)\omega_{n-1}} \int_{S_r} (g_{ij,i} - g_{ii,j}) \nu_j \; dS_r \\
    &= \lim_{r\to\infty} \frac{1}{2(n-1)\omega_{n-1}} \int_{S_r} (f_{ii}f_j + f_{ij}f_i - 2f_{ij}f_i) \nu_j \; dS_r \\
    &= \lim_{r\to\infty} \frac{1}{2(n-1)\omega_{n-1}} \int_{S_r} (f_{ii}f_j - f_{ij}f_i) \nu_j \; dS_r.
\end{align*}
By the asymtotic flatness assumption, the function $1/(1+|\nabla
f|^2)$ goes to 1 at infinity. Hence we can alternately write the
mass as
\begin{equation*}
    m = \lim_{r\to\infty} \frac{1}{2(n-1)\omega_{n-1}} \int_{S_r} \frac{1}{1+|\nabla f|^2} (f_{ii}f_j - f_{ij}f_i) \nu_j \; dS_r.
\end{equation*}
Now apply the divergence theorem in $(\mathbb R^n,\delta)$ and use
Lemma 1 to get
\begin{align*}
    m &= \frac{1}{2(n-1)\omega_{n-1}} \int_{\mathbb R^n} \nabla \cdot \left( \frac{1}{1+|\nabla f|^2} (f_{ii}f_j - f_{ij}f_i)\partial_j \right) \; dV_{\delta} \\
    &= \frac{1}{2(n-1)\omega_{n-1}} \int_{\mathbb R^n} R \; dV_{\delta} \\
    &= \frac{1}{2(n-1)\omega_{n-1}} \int_{M^n} R \frac{1}{\sqrt{1+|\nabla f|^2}} \; dV_g
\end{align*}
since
\begin{equation*}
    dV_g = \sqrt{\det(g)} \; dV_{\delta} = \sqrt{1+|\nabla f|^2} \; dV_{\delta}.
\end{equation*}
\end{proof}

We end this introduction by noting that Lam has generalized the
identity in lemma \ref{LamLemma} beyond $\real^n$, as is natural to
do, as described in \cite{LamThesis,LamPaper}.  It is of great interest to see
what other applications this identity may have.

\section{A Trio of Inequalities}\label{sec:trio}

In this section we describe the Riemannian positive mass inequality,
the Riemannian Penrose inequality, and the Riemannian ZAS
inequality. These three inequalities are intimately connected. Most
obviously, the Schwarzschild metrics (of which there are precisely
three up to scaling) are the cases of equality for these three
inequalities.

In dimension two, the only reasonable interpretation of these three
inequalities follows from the Gauss-Bonnet theorem (a natural story
which we will not tell here). This is not surprising since the
Gauss-Bonnet theorem is a global result about scalar curvature
(equal to twice the Gauss curvature of a surface) as well. Hence,
these theorems are, in a very real sense, nontrivial generalizations
of the Gauss-Bonnet theorem to higher dimensions. Since the two
dimensional case is somewhat exceptional and completely understood,
we will always assume that $n \ge 3$.

All three inequalities are known to be true in dimension $n \le 7$
\cite{PMT2,BrayLee,BrayJauregui} (assuming an interesting open
conjecture for the Riemannian ZAS inequality \cite{BrayJauregui},
discussed later). Also, the only known proof of the Riemannian ZAS
inequality in dimension $n \le 7$ (modulo that interesting open
conjecture \cite{BrayJauregui}) requires the Riemannian Penrose
inequality in dimension $n \le 7$ \cite{BrayLee}, which in turn so
far has only been proved using the Riemannian positive mass
inequality in dimension $n \le 7$ \cite{PMT2}.  Hence, there appears
to be a deep connection between these three inequalities.

\subsection{The Riemannian Positive Mass Inequality}

The Riemannian positive mass inequality is an elegant statement
which reveals the global effect of nonnegative scalar curvature on
an asymptotically flat manifold.  Since scalar curvature is
proportional to energy density in this context, the inequality
equivalently states that nonnegative energy density implies
nonnegative total mass.  It also states that there is a unique zero
total mass metric with nonnegative energy density, namely the flat
metric on $\real^n$.

\begin{RPMI}
Let $(M^n,g)$, $n \ge 3$, be a complete, asymptotically flat, smooth
$n$-manifold with nonnegative scalar curvature and total mass $m$.
Then
\begin{equation*}
    m\geq 0
\end{equation*}
with equality if and only if $(M^n,g)$ is isometric to $\mathbb R^n$
with the standard flat metric.
\end{RPMI}

Schoen and Yau surprised the relativity world in 1979 when they
proved the Riemannian positive mass inequality in dimension three
\cite{PMT1} because the techniques they used were based on existence
and stability properties of minimal surfaces.  A nice discussion of
these techniques, as well as related results including a fascinating
discussion of the hyperbolic version of the positive mass theorem,
is provided in \cite{GallowaySurvey}, also in this book dedicated to
Rick Schoen's 60th birthday. In \cite{PMT2}, Schoen and Yau extended their argument
using Jang's equation to prove the more general positive mass inequality for a generic slice of
a spacetime, in dimension three.  In addition, Schoen and Yau showed
that their argument extended to dimensions less than eight by using
a minimal hypersurface argument which is inductive on dimension
\cite{SchoenYauCompact, SchoenVariational}.  The higher dimension Jang equation argument which
proves the higher dimensional positive mass inequality in dimensions less than eight
is treated very carefully in
\cite{EichmairJang37}, following the method suggested by \cite{PMT2}.
However, another initially surprising fact is that minimal
hypersurfaces may have singularities on a subset whose codimension
has been proven to be at least seven
\cite{SimonGMT,FedererGMT,SimonsCone1,SimonsCone2}, which prevented
\cite{SchoenYauCompact,SchoenVariational} from proving the Riemannian positive mass inequality for
manifolds in dimension eight and higher. However, in 2009 Schoen has
been giving talks \cite{SimonsTalk-Schoen} about an approach to
overcoming this obstacle by trying to show that the minimal
hypersurface singularities can be sufficiently controlled. If this
approach is ultimately successful, this would prove the Riemannian
positive mass inequality in all dimensions.

In addition, Lohkamp has studied a somewhat different way of
modifying the original Schoen-Yau minimal surface argument.  He
announced a proof of the positive mass inequality in all dimensions
around or before 2004 which he describes in a 2006 preprint
\cite{LohkampPMT}. The techniques involved have been further expanded
upon in followup preprints \cite{Lohkamp1,Lohkamp2}.

Meanwhile, Witten's 1981 proof \cite{Witten, ParkerTaubes} of the
positive mass inequality works in all dimensions, but only for
manifolds which are spin.  So far no one has been able to get around
this extra topological assumption using a spinor approach.

Also, Lam's results \cite{LamThesis,LamPaper}, described at the end of the
previous section, work in general dimension and imply both the
Riemannian positive mass inequality and the Riemannian Penrose
inequality for certain manifolds which can be realized as graphs
over $\real^n$.  Naturally this assumption is very restrictive, but
the results are very good for helping to develop intuition.

Finally, in dimension three, the inverse mean curvature flow
techniques of Huisken and Ilmanen \cite{HuiskenIlmanen} prove the
Riemannian positive mass inequality by starting the inverse mean
curvature flow on a surface which is a small geodesic sphere around
a fixed point, in the limit as the radius of the geodesic sphere
goes to zero.  In fact, Lemma 8.1 of \cite{HuiskenIlmanen} proves
that there is a solution to inverse mean curvature flow starting at
a point, from which the Riemannian positive mass inequality in
dimension three follows.  Huisken and Ilmanen show that their
inverse mean curvature flow exists in dimensions less than eight.
However, the monotonicity of the Hawking mass of the surfaces in the
flow, a critical component of their proof of the Riemannian positive mass
inequality in dimension three, is only always satisfied in dimension
three. It is a fascinating open question whether or not this type of
approach could work in dimensions greater than three.


In \cite{LohkampCompact}, Lohkamp makes a nice connection between
asymptotically flat manifolds and compact manifolds in the context
of scalar curvature by proving the following:

\begin{lemma}\cite{LohkampCompact}
Let $(M^n,g)$ be a complete, asymptotically flat, smooth manifold,
and let $\tilde{M}^n$ be the compact smooth manifold which results
from compactifying $M^n$ by adding a point at infinity.  If $T^n \#
\tilde{M}^n$ does not admit a metric of positive scalar curvature,
then the Riemannian positive mass inequality is true for $(M^n,g)$.
\end{lemma}

To be clear, we mean to say that if $(M^n,g)$ has nonnegative scalar
curvature, then its total mass is nonnegative.  The case of equality
when the total mass is zero is also true by a short time Ricci flow / conformal change
argument \cite{PMT1, Bartnik} which we do not
present here.

Hence, the Riemannian positive mass inequality may be approached by
studying topological obstructions to compact manifolds admitting
metrics of positive scalar curvature.  We refer the reader to
\cite{GallowaySurvey} for more discussion. Besides this topological
question being interesting in its own right, the Riemannian positive
mass theorem in all dimensions would follow from proving that $T^n
\# M^n$ does not admit a metric of positive scalar curvature for any
compact $M^n$.

Lohkamp's lemma may be proved as follows.  Let $(M^n,g)$ be
complete, smooth, and asymptotically flat with nonnegative scalar
curvature. Suppose that the total mass $m < 0$.  Then Lohkamp shows
that the metric $g$ may be perturbed to a new metric $\tilde{g}$
while preserving nonnegative scalar curvature such that outside a
bounded set $\tilde{g}$ is flat. Choosing a cube which contains this
bounded set and identifying the opposite sides produces a metric
$\tilde{\tilde{g}}$ on $T^n \# \tilde{M}^n$ with nonnegative scalar
curvature. Furthermore, the perturbation process produces a region
where the scalar curvature of $\tilde{\tilde{g}}$ is strictly
positive. Multiplying $\tilde{\tilde{g}}$ through by the correct
conformal factor close to one spreads out this positive scalar
curvature so that the new conformal metric has positive scalar
curvature everywhere.  But $T^n \# \tilde{M}^n$ does not admit a
metric of positive scalar curvature, so we have a contradiction.
Hence, we can not have $m < 0$, so $m \ge 0$.

For the case of equality, suppose $m = 0$.  Then the previously
mentioned short time Ricci flow / conformal change argument \cite{PMT1} decreases the total mass to
being negative unless the manifold is Ricci flat.  The fact that Ricci flat plus asymptotically
flat implies flat is in \cite{Bartnik}.  Since the total mass can not be negative by the above
paragraph, it must be that $(M^n,g)$ is isometric to the flat metric
on $\real^n$.

The key step in the above proof is to take the asymptotically flat
metric $g$ with negative total mass and to perturb it while
preserving nonnegative scalar curvature to get the metric
$\tilde{g}$ which is flat outside a bounded set.  Besides Lohkamp's
construction \cite{LohkampCompact}, one can also use one of the author's
thesis results \cite{BrayThesis} (based on the harmonically flat perturbation result
\cite{PMT2} of Schoen and Yau) that allows an asymptotically flat
manifold to be perturbed to being Schwarzschild at infinity while
preserving nonnegative scalar curvature and changing the total mass
as little as one likes, so that in this case the total mass stays
negative. Since Schwarzschild is spherically symmetric, the metric
can then be ``curved up'' to being precisely flat outside a bounded
set in the space of spherically symmetric metrics with nonnegative
scalar curvature, which is most
easily seen by considering the Hawking masses of the spherically
symmetric spheres.

\subsection{The Riemannian Penrose Inequality}\label{RPI}

The Riemannian Penrose inequality is another elegant statement
which, like the Riemannian positive mass inequality, reveals the
global effect of nonnegative scalar curvature on an asymptotically
flat manifold, but where now there may be an interior boundary.
Conditions imposed on this interior boundary allow each connected
component of the boundary to be interpreted physically as the
apparent horizon of a black hole.  As before, scalar curvature is
proportional to energy density at each point.  Now suppose that we
define the mass contributed by this collection of black holes in
terms of the $(n-1)$ volume of the boundary.  In doing so, we may
interpret the Riemannian Penrose inequality as the physical
statement that nonnegative energy density implies that the total
mass is at least the mass contributed by the black holes.

\begin{RPI}
Let $(M^n,g)$, $n \ge 3$, be a complete, asymptotically flat, smooth
$n$-manifold with nonnegative scalar curvature, total mass $m$, and
a strictly outerminimizing smooth minimal boundary which is compact
(but not necessarily connected) of total $(n-1)$ volume $A$. Then
\begin{equation}
    m\geq \frac12 \left( \frac{A}{\omega_{n-1}} \right)^{\frac{n-2}{n-1}}
\end{equation}
with equality if and only if $(M^n,g)$ is isometric to the region of
a Schwarzschild metric of positive mass outside its minimal
hypersurface.
\end{RPI}

In the above statement, $\omega_{n-1}$ is the measure of the unit
$(n-1)$ sphere in $\real^n$.  A minimal boundary is one which has
zero mean curvature.  Finally, strictly outerminimizing means that
all other smooth hypersurfaces in the same homology class have
greater $(n-1)$ volume.

We hasten to add that it may be of considerable interest and
importance to study versions of the above statement where the
smoothness of the boundary is removed, and the notion of being
minimal is revisited. For example, in dimensions eight and higher,
minimizing hypersurfaces may have singularities on a subset whose
codimension has been proven to be at least seven
\cite{SimonGMT,FedererGMT,SimonsCone1,SimonsCone2}, so a more
general notion of allowable boundaries could be very important.

The Riemannian Penrose inequality was effectively conjectured by
Penrose \cite{Penrose} in 1973.  The inequality was proved in
dimension less than eight \cite{BrayLee} in 2007 by Lee and the
author (both former students of Rick Schoen, it should be noted in
this birthday volume) by generalizing the conformal flow of metrics method
originally used in the author's 1999 proof of the inequality in
dimension three \cite{BrayPenrose}.  The difficulty in pushing the
argument through to dimensions greater than seven is, once again,
that minimal hypersurfaces may have singularities on a subset whose
codimension has been proven to be at least seven
\cite{SimonGMT,FedererGMT,SimonsCone1,SimonsCone2}. Proving the
Riemannian Penrose inequality in dimensions greater than seven is an excellent
open problem.

We refer the reader to \cite{BrayNotices} for a general discussion
of the conformal flow of metrics technique.  The basic idea is to
flow the starting manifold $(M^n,g)$ to a Schwarzschild metric while
preserving nonnegative scalar curvature in
such a way that the $(n-1)$ volume of the boundary remains fixed and
the total mass is nonincreasing.  The Riemannian Penrose inequality
is defined to give equality for the Schwarzschild metrics, so the
inequality for the original starting manifold follows.  The
surprising fact is that it is possible to do this flow inside the
conformal class of the original manifold, which could even have a
different topology than the Schwarzschild metric.  The solution is
to allow the boundary to be a moving boundary, where in the limit as
the flow time goes to infinity, the moving boundary (which sometimes
jumps over regions and topology) goes to infinity as well,
eventually enclosing any bounded set. Another important
characteristic of the conformal flow of metrics approach is that the
monotonicity of the total mass is, quite beautifully, a direct and
yet nontrivial consequence of the Riemannian positive mass theorem,
as described in \cite{BrayPenrose,BrayNotices}.

An earlier breakthrough on the Riemannian Penrose inequality was due
to Huisken and Ilmanen \cite{HuiskenIlmanen} in 1997 using a
technique called inverse mean curvature flow, which proved the
inequality in dimension three when the minimal boundary is connected
and outermost (not enclosed by any other minimal surface). Inverse
mean curvature flow is of great independent interest as well. In the
seventies, Geroch \cite{Geroch} and Jang and Wald \cite{JangWald}
proposed this novel approach to address the Riemannian Penrose
inequality by flowing surfaces with speed equal to the reciprocal of
their mean curvatures at each point. Amazingly, there is an explicit
quantity called the Hawking mass which is nondecreasing under this
flow, in dimension three, and which equals the right and left sides
of the Riemannian Penrose inequality at $t=0$ and $t=\infty$,
respectively. However, Geroch, Jang, and Wald did not prove a
general existence theory for this flow, which is understandable
because there are examples of manifolds and initial surfaces where
it is known that smooth solutions to inverse mean curvature flow do
not exist. Hence, it was quite surprising and a true breakthrough in
1997 when Huisken and Ilmanen found a weak notion of inverse mean
curvature flow which allowed the surfaces to occasionally jump while
still preserving the monotonicity of the Hawking mass
\cite{HuiskenIlmanen}.

The monotonicity of the Hawking mass turns out to rely on the two
dimensional Gauss-Bonnet formula for the surfaces in the flow which
need to have Euler characteristic not exceeding two.  Consequently,
the argument so far only works in dimension three.  It is a
fascinating open question whether or not this type of approach could
work in dimensions greater than three. However, by a result due to
Meeks, Simon, and Yau \cite{MSY}, the exterior region outside the
outermost minimal surface of an asymptotically flat 3-manifold is,
quite remarkably, always diffeomorphic to $\real^3$ minus a finite
number of disjoint closed balls.  Huisken and Ilmanen start their
inverse mean curvature flow on the boundary of one of these balls so
that the initial Euler characteristic equals two, and then use the
trivial topology of the exterior region to guarantee that the
surfaces resulting from their weak inverse mean curvature flow stay
connected, and hence continue to have Euler characteristic not
exceeding two.  The Riemannian Penrose inequality follows, where $A$
is the area of any of the connected components of the outermost
minimal surface of the manifold. We refer the reader to
\cite{HuiskenIlmanen,BrayNotices} for more discussion of
inverse mean curvature flow.

In dimension greater than eight,
Lam's proof of the Riemannian Penrose inequality for manifolds which are graphs over $\real^n$
is all of
the cases where the Riemannian Penrose inequality is known to be
true, as far as the author is aware.
Previously the inequality was only known in the spherically symmetric case in
these higher dimensions, which may also be realized as graphs in Lam's examples.

To push the physical interpretation of the Riemannian Penrose
inequality a little further, let $\Sigma_{BH}$ be the minimal
boundary with connected components $\{\Sigma_{BH}^i\}$, each of
which we will think of as the apparent horizon of a black hole
(which is the case when $\Sigma$ is the outermost minimal
hypersurface).  Then if we define the mass contributed by this
collection of black holes to be
\begin{equation}\label{eqn:bh}
m_{BH} = \frac12 \left( \frac{|\Sigma_{BH}|}{\omega_{n-1}}
\right)^{\frac{n-2}{n-1}},
\end{equation}
where $|\Sigma_{BH}|$ is the $(n-1)$ volume of the boundary
$\Sigma_{BH}$, then the Riemannian Penrose inequality may be stated
as
\begin{equation}
   m \ge m_{BH},
\end{equation}
that the total mass is at least the mass contributed by the black
holes.

We also refer the reader to \cite{Mars} for Marc Mars's excellent
survey of results on the Penrose inequality, a more general
statement about spacelike hypersurfaces $M^n$ of spacetimes
$(N^{n+1},g_N)$ involving the second fundamental form $k$ and the
induced metric $g$ on $M^n$.  We note that when the second
fundamental form $k$ is zero, the relevant data is $(M^n,g)$, a
Riemannian manifold. This observation inspired Huisken and Ilmanen
\cite{HuiskenIlmanen} to refer to the zero second fundamental form
case of the Penrose inequality as the Riemannian Penrose inequality,
by which it has been known ever since. The full Penrose inequality,
or Penrose conjecture, is a very important open problem since it
says something very interesting about the geometries of spacetimes.
In \cite{BrayKhuri1,BrayKhuri2}, Khuri and the author propose an approach to
reduce the full Penrose inequality (which is still wide open) to the
Riemannian Penrose inequality (which has been proven in dimensions
less than eight). The key step is the derivation of a new identity,
called the generalized Schoen-Yau identity, which is a nontrivial generalization of an
identity first proved by Schoen and Yau in their proof of the
positive mass inequality \cite{PMT2}.  We stop here because, while
it would have been very tempting to survey recent progress on the
Penrose inequality for this volume, Marc Mars's recent survey
\cite{Mars} makes this completely unnecessary.

\subsection{The Riemannian ZAS Inequality}

The Riemannian ZAS inequality is the last of our trio of
inequalities.  Like the Riemannian Penrose inequality and the
Riemannian positive mass inequality, the Riemannian ZAS inequality
describes the global effect of nonnegative scalar curvature on an
asymptotically flat manifold.  The new feature, however, is that our
manifold is now allowed to have a finite number of isolated
singularities.  In \cite{BrayJauregui}, Jauregui and the author
define a notion of the mass contributed by a finite collection of
zero area singularities (ZAS), which is very general. The total mass
of an asymptotically flat manifold with nonnegative scalar curvature
is then shown to be at least the mass contributed by the
singularities, whenever a certain geometric conjecture is true.  We
refer the interested reader to \cite{BrayJauregui} for the details
but survey some of the main points from that paper here.  We comment
that \cite{BrayJauregui} focused on $n=3$, whereas we are presenting
the general dimension case here (with Jauregui's help), which is a
natural generalization.  We also comment that of the three
inequalities, the Riemannian ZAS inequality is the least well
understood with many excellent problems left open and plenty of
interesting directions to explore.

Another interesting point is that the Riemannian ZAS inequality is
primarily motivated by geometric considerations as opposed to
physical ones.  While well established physical reasoning led to the
conjecturing of the positive mass inequality and the Penrose
inequality, the Riemannian ZAS inequality, as we pose it here, has
simply been defined to be what can be proved mathematically so far.
In fact, a precise statement of what the ZAS inequality should be
exactly has not yet been established beyond the Riemannian case
(although there are some natural candidates), and even the
Riemannian case is not perfectly clear. Hence, it is not yet evident
what the physical implications of this inequality, if any, will be,
although the discussion of singularities in general relativity
is well underway \cite{Gibbons,Dotti}.

As promised, the Schwarzschild metrics with negative mass are cases
of equality of the Riemannian ZAS inequality.  Metrically, these
singularities look like a cusp at a single point, as in figure
\ref{NegativeSchwarzschild}.  Note that the singularity is not
roughly conical in the usual sense as the picture suggests since the
background metric is Minkowski space and the hypersurface is
becoming null at the singularity. Also, referring back to equation
\ref{eqn:conformalform}, one sees that the singularity of a negative
mass Schwarzschild metric occurs on the coordinate sphere of radius
$r = (|m|/2)^\frac{1}{n-2} > 0$ since the conformal factor goes to
zero on this sphere. In fact, it is useful to think of these
singularities as surfaces with zero area.

\begin{definition} \cite{BrayJauregui}
Let $M^n$ be a smooth $n$-manifold with smooth compact boundary
$\partial M$ and let $g$ be a complete, smooth, asymptotically flat metric
on $M \backslash
\partial M$.  A connected component $\Sigma$ of $\partial M$ is a
\textbf{zero area singularity (ZAS)} of $(M^n,g)$ if for every
sequence of surfaces $\{S_n\}$ converging in $C^1$ to $\Sigma$, the
areas of $S_n$ measured with respect to $g$ converge to zero.
\end{definition}

In this section we will consider the case when every connected
component of $\partial M$ is a ZAS, so we define $\Sigma_{ZAS} =
\partial M$.  As one might expect, not all singularities have equally bad
behavior, even when curvatures are blowing up to infinity. In
\cite{BrayJauregui}, notions of regular ZAS, harmonically regular
ZAS, and globally harmonically regular ZAS are defined.

\begin{definition} \cite{BrayJauregui}
Let $\Sigma$ be a ZAS of $(M^n,g)$.  Then $\Sigma$ is
\textbf{regular} if there exists a smooth, nonnegative function
$\overline\varphi$ and a smooth metric $\overline{g}$, both defined
on a neighborhood $U$ of $\Sigma$ (which of course includes
$\Sigma$), such that

(1) $\overline\varphi$ vanishes precisely on $\Sigma$,

(2) $\overline\nu(\overline\varphi) > 0$ on $\Sigma$, where
$\overline\nu$ is the inward unit normal to $\Sigma$ with respect to
$\overline{g}$,

(3) $g = \overline\varphi^{\,4/(n-2)} \overline{g}$ on $U \backslash
\Sigma$.

\noindent If such a pair $(\overline{g},\overline\varphi)$ exists,
it is called a \textbf{local resolution} of $\Sigma$.
\end{definition}

When $\overline\varphi$ is also harmonic with respect to
$\overline{g}$, then the ZAS is defined to be harmonically regular
with a local harmonic resolution. Finally, if all of the connected
components of $\partial M$ are harmonically regular ZAS with the
same $\overline\varphi$, where $U$ is the entire manifold, then all
of the ZAS are defined to be globally harmonically regular with a
global harmonic resolution. Note that the negative mass
Schwarzschild metric singularity is globally harmonically regular
(and hence harmonically regular and regular), as can be seen from
equation \ref{eqn:conformalform}.  However, there are regular ZAS
which are not harmonically regular (section 4.4 of
\cite{BrayJauregui}).

These notions of regular ZAS are useful because generic ZAS may be
approximated by them, as described in \cite{BrayJauregui}. In doing
so, the definition of the mass of a regular ZAS may be extended to
give a definition of the mass of a generic ZAS.  We refer the reader
to \cite{BrayJauregui} for these and other discussions.

\begin{definition}
Suppose $\Sigma$ is a regular ZAS of $(M^n,g)$ with local resolution
$(\ol g, \ol \varphi)$.  We define the \textbf{regular mass} of
$\Sigma$ to be
\begin{equation}\label{eqn:defregmass}
m_{reg}(\Sigma) = - \frac{2}{(n-2)^2} \left(\frac{1}{\omega_{n-1}}
\int_\Sigma ({\ol \nu} (\ol \varphi))^{\frac{2(n-1)}{n}} \;\ol{dA}
\right)^{\frac{n}{n-1}},
\end{equation}
where ${\ol \nu}$ is the inward unit normal to $\Sigma$ with respect
to $\ol g$, $\ol{dA}$ is the $(n-1)$ volume form of $\ol g$, and
$\omega_{n-1}$ is the $(n-1)$ volume of the unit $(n-1)$ sphere in
$\real^n$.
\end{definition}

Again, this definition of the mass of a regular ZAS was chosen
primarily because theorems can be proved about it.  Hence, other
definitions should also be considered.  However, in addition to
being able to prove theorems using this definition, it also gives
the correct answer in the case of a negative mass Schwarzschild
singularity, namely $-|m|$, is independent of the particular choice
of local resolution, and only depends on the local geometry near the
singularity. The local nature of this definition of the mass of a
regular ZAS is best seen by this next lemma (proposition 12 in
\cite{BrayJauregui}). In fact, one could take the following formula
to be the definition of the mass of a regular ZAS, if one prefers.
\begin{lemma}
   Suppose $\Sigma$ is a regular ZAS of $(M^n,g)$.  If
   $\{\Sigma_n\}$ is a sequence of surfaces converging to $\Sigma$
   in $C^2$, then
\begin{equation}\label{eqn:limit}
   m_{reg}(\Sigma) = \lim_{n \rightarrow \infty} - \frac{1}{2(n-1)^2} \left(\frac{1}{\omega_{n-1}}
\int_{\Sigma_n} H^{\frac{2(n-1)}{n}} \; dA \right)^{\frac{n}{n-1}},
\end{equation}
where $H$ is the mean curvature of $\Sigma_n$, and $dA$ is the
$(n-1)$ volume form on $\Sigma_n$.
\end{lemma}

We define the mass contributed by a collection of regular ZAS using
equations \ref{eqn:defregmass} and \ref{eqn:limit}, where $\Sigma =
\Sigma_{ZAS} = \partial M$ may now have multiple components.  That
is, we define
\begin{equation}\label{eqn:zas}
   m_{ZAS} = m_{reg}(\Sigma_{ZAS}).
\end{equation}
Again, we refer the reader to \cite{BrayJauregui} for the case of
generic ZAS.  In all cases, the goal is to prove that $m \ge
m_{ZAS}$, that is, that the total mass of a complete, asymptotically
flat, smooth $n$-manifold with smooth boundary, each connected
component of which is a ZAS, is at least the mass contributed by the
ZAS.

\begin{RZASI}
Let $(M^n,g)$, $n \ge 3$, be a complete, asymptotically flat, smooth
$n$-manifold with nonnegative scalar curvature, total mass $m$, and
nonempty compact boundary $\partial M$ (smooth with respect to the
smooth structure), each connected component of which is a ZAS. Then
\begin{equation}
   m \ge m_{ZAS}
\end{equation}
with equality if and only if $(M^n,g)$ is isometric to a
Schwarzschild metric of negative mass.
\end{RZASI}

The Riemannian ZAS inequality in dimension three
is proven in \cite{BrayJauregui} whenever a certain geometric
conjecture, called the conformal conjecture, is true.  The proof in \cite{BrayJauregui} generalizes
to dimensions less than eight without any significant changes.  We state the
conformal conjecture more generally than in \cite{BrayJauregui} to emphasize what is needed
to be true for the method in \cite{BrayJauregui} to work.

\begin{CC}
Let $(M^n,g)$, $n \ge 3$, be a complete, asymptotically flat, smooth
$n$-manifold with nonempty compact boundary $\Sigma = \partial M$
which is smooth with respect to the smooth structure and the metric
$g$. Then there exists a positive harmonic function $u(x)$ on
$(M^n,g)$ going to one at infinity such that if we let $\ol{g} =
u(x)^{4/(n-2)} g$ and $\tilde\Sigma$ be the outermost minimal
enclosure of $\Sigma$ in $(M^n,\ol{g})$, then

(1) $\tilde\Sigma$ and $\Sigma$ have the same $(n-1)$ volume in
$(M^n,\ol{g})$,

(2) $\tilde\Sigma$ has zero mean curvature, at least in some weak
sense compatible with the Riemannian Penrose inequality.
\end{CC}

The power of the conformal conjecture, when it is true, is that it allows one to apply the Riemannian
Penrose inequality to any asymptotically flat manifold with nonnegative scalar curvature and
compact boundary to gain some
information about the manifold.  If the original metric has
nonnegative scalar curvature, then since
\begin{equation}\label{eqn:confscalarcurv2}
   \ol{R} = u(x)^{-\frac{n+2}{n-2}}\left( -\frac{4(n-1)}{n-2} \Delta u + Ru \right),
\end{equation}
the new metric $\ol{g}$ will have nonnegative scalar curvature as well.  It turns out to be too much
to ask that the boundary $\Sigma$ satisfies the conditions of the Riemannian Penrose inequality (being
a strictly outerminimizing minimal surface).  Instead, the conjecture simply asks $\Sigma$ to have the
same $(n-1)$ volume as its outermost minimal enclosure $\tilde{\Sigma}$,
which is automatically strictly outerminimizing
by definition.  Hence, if the conjecture is satisfied, and $\tilde{\Sigma}$ has zero mean curvature,
at least in some weak sense compatible with the Riemannian Penrose inequality, then we can effectively
apply the Riemannian Penrose inequality to $\Sigma$, since it has the same $(n-1)$ volume as
$\tilde\Sigma$.  This inequality then gives information about the original manifold $(M^n,g)$,
for example proving the Riemannian ZAS inequality as described in \cite{BrayJauregui} (when combined
with some additional arguments), when the
ZAS are globally harmonically regular.  A process by which generic ZAS are approximated by
globally harmonically regular ZAS then proves the Riemannian ZAS inequality in dimension less
than eight, whenever the conformal conjecture is true.

We note that the conformal conjecture is true when $(M^n,g)$ is spherically symmetric, which is a good
exercise.  Also, a substantial amount of progress has been made on the conformal conjecture in
Jauregui's thesis \cite{JaureguiThesis},
which proves that a conformal factor may be chosen so that $\tilde\Sigma$ has zero
mean curvature off a set of measure zero.  As we commented before, it may be important to study
versions of the Riemannian Penrose inequality where the smoothness of the boundary is not required.
Finding the most general allowable boundary conditions for the Riemannian Penrose inequality
could be a very interesting and important open problem, potentially with implications for
the Riemannian ZAS inequality.

In the case of a single ZAS in dimension three, Huisken and Ilmanen's
inverse mean curvature flow proves the Riemannian ZAS inequality, as shown by Robbins in
\cite{Robbins, RobbinsThesis}.  However,
if there is more than one ZAS, then no conclusion results from inverse mean curvature flow.  The reason
is that the Hawking mass is only nondecreasing if the surfaces in the flow are connected.  Hence, the
flow must start at one of the ZAS.  However, when the surfaces flow over the other ZAS, the Hawking
mass will not be nondecreasing.  Once could hope that the inverse mean curvature flow might always
jump over the other ZAS, but this is not the case, since we know that the total mass can in fact be
less than the mass of any single ZAS.  More realistically, one could hope to estimate how
much the Hawking mass decreases when the surfaces flow over the other ZAS.  This could be an interesting
problem to study.  It is even possible that the surfaces in the flow get ``snagged'' on ZAS
for a positive amount of time in the flow when they pass over
ZAS, which would be another interesting phenomenon to study.

Also, when $(M^n,g)$ is isometric to a graph over $\real^n$ in $(n+1)$ dimensional Minkowski space,
Lam's methods \cite{LamThesis,LamPaper}
prove a version of the ZAS inequality when ZAS are present, as in figure \ref{NegativeSchwarzschild}.
This argument works for any number of ZAS in any dimension.

We comment that equations \ref{eqn:bh} and \ref{eqn:zas} combine to
motivate defining the quasi-local mass functional
\begin{equation}
  {\mathcal M}(S) = \frac12 \left( \frac{|S|}{\omega_{n-1}}
   \right)^{\frac{n-2}{n-1}}
   \;-\;\; \frac{1}{2(n-1)^2} \left(\frac{1}{\omega_{n-1}}
\int_{S} H^{\frac{2(n-1)}{n}} \; dA \right)^{\frac{n}{n-1}}
\end{equation}
for a compact hypersurface $S$, where $|S|$ is the $(n-1)$ volume of
$S$, $H$ is the mean curvature of $S$, and $dA$ is the $(n-1)$
volume form on $S$.  The virtue of this quasi-local mass functional
${\mathcal M}$ is that on a collection of minimal hypersurfaces (where $H=0$)
it gives the mass contributed by a collection of black holes as in
equation \ref{eqn:bh} and in the limit as $S$ converges to a
collection of ZAS in $C^2$ (where $|S|$ is going to zero) it gives
the mass contributed by a collection of ZAS as in equation
\ref{eqn:zas}.  Furthermore, ${\mathcal M}(S) = m$ for any of the spherically
symmetric spheres $S$ of a Schwarzschild metric of mass $m$. However, the author is not aware
of any flow which makes this functional nondecreasing.  The big open problem is to try to
generalize the successes of inverse mean curvature flow and the fact that it makes the Hawking mass
functional nondecreasing to general dimension and to surfaces which are not necessarily connected.

Another interesting conjecture to study, again motivated by equations \ref{eqn:bh} and \ref{eqn:zas},
is one which combines the Riemannian Penrose inequality with the Riemannian ZAS inequality, which we
will state as
\begin{equation}
m \ge m_{BH} + m_{ZAS}.
\end{equation}
In the above combined black hole ZAS conjecture,
we are assuming that every connected component of the boundary of the
smooth topological manifold $M^n$ is either a strictly outerminimizing minimal hypersurface
(corresponding to the apparent horizon of a black hole) or a ZAS.  Note that the masses of black holes
do not ``add'' in that there are examples of asymptotically flat manifolds where
\begin{equation}
   m < \frac12 \left( \frac{|\Sigma_{BH}^1|}{\omega_{n-1}}\right)^{\frac{n-2}{n-1}}
     + \frac12 \left( \frac{|\Sigma_{BH}^2|}{\omega_{n-1}}\right)^{\frac{n-2}{n-1}},
\end{equation}
where the boundary of the manifold is a strictly outerminimizing minimal hypersurface with
connected components $\Sigma_{BH}^1$ and $\Sigma_{BH}^2$.  Similarly, the masses of ZAS do not
``add'' either.  Essentially, the reason is related to the fact that the potential energy between two
objects in Newtonian physics is $-Gm_1m_2/r^{n-2}$, which is negative when the two masses have the same
sign.  However, this potential energy term is positive when the two masses have opposite signs.
Inspired by this, the above conjecture that the total mass is at least
the sum of the mass contributed by the collection
of black holes with the mass contributed by the collection of ZAS seems plausible, and hence
deserving of further study.  Note that another way to state this conjecture, using the quasi-local mass
function just defined, is as $m \ge {\mathcal M}(\partial M)$.  Conjecturally, the cases of equality
would be the Schwarzschild metrics of any mass, positive or negative.
In addition, the metric on $\real^n$ (minus some compact set) given by
\begin{equation}
   g_{ij} = \left(1 + \frac{m_1}{2 |x - x_1|^{n-2}}
                    + \frac{m_2}{2 |x - x_2|^{n-2}} \right)^\frac{4}{n-2} \delta_{ij}
\end{equation}
in the limit as $x_1, x_2 \in \real^n$ get infinitely far apart, are arbitrarily close to being
cases of equality, when $m_1$ and $m_2$ have opposite signs, corresponding to precisely one black hole
and one ZAS.  This last observation could be an important hint for approaching this combined
black hole ZAS conjecture.

The author would like to thank Michael Eichmair, Jeff Jauregui, George Lam, Dan Lee, and
Fernando Schwartz for helpful suggestions with this paper, all of whom join the author in wishing
Rick a very happy 60th birthday.


\end{document}